\documentclass{amsart}
\usepackage[hmargin=3.5cm,vmargin=3.5cm]{geometry}

\usepackage[utf8]{inputenc}

\usepackage{amsmath,amssymb,amsthm, mathtools}

\usepackage{graphicx}

\usepackage{enumerate,blindtext}
\usepackage[linktocpage=true]{hyperref}

\usepackage{todonotes, lipsum}
\usepackage[all,cmtip]{xy}

\makeatletter
\g@addto@macro\bfseries{\boldmath}
\makeatother

\hypersetup{
bookmarksnumbered=true
}

\newtheorem{theorem}{Theorem}
\newtheorem{lemma}{Lemma}[section]

\newtheorem{proposition}[lemma]{Proposition}
\theoremstyle{definition}

\newtheorem{remark}[lemma]{Remark}
\newtheorem{definition}[lemma]{Definition}

\newtheorem*{korollar*}{Corollary}

\newcommand{\G}{\ensuremath{\Gamma}}
\newcommand{\g}{\ensuremath{\gamma}}

\newcommand{\C}{\ensuremath{\mathbb{C}}}

\newcommand{\bbS}{\ensuremath{\mathbb{S}}}
\newcommand{\M}{\ensuremath{\mathbb{M}}}
\newcommand{\N}{\ensuremath{\mathbb{N}}}

\newcommand{\ov}{\ensuremath{\overline}}
\newcommand{\mf}{\ensuremath{\mathfrak}}
\newcommand{\mc}{\ensuremath{\mathcal}}
\newcommand{\Z}{\ensuremath{\mathbb{Z}}}
\newcommand{\R}{\ensuremath{\mathbb{R}}}
\newcommand{\inv}{\ensuremath{^{-1}}}
\newcommand{\norm}[1]{\left\|#1\right\|}

\renewcommand{\d}[1][t]{\ensuremath{\left.\frac{d}{d#1}\right|_{#1=0}}}
\renewcommand{\Re}{\ensuremath{\operatorname{Re}}}

\DeclareMathOperator{\interior}{int}
\DeclareMathOperator{\ext}{ext}

\DeclareMathOperator{\sign}{sign}

\DeclareMathOperator{\dom}{dom}
\title[Spectral Asymptotics for Kinetic Brownian Motion]{Spectral Asymptotics for Kinetic Brownian Motion on  Surfaces of Constant Curvature}

\author{Martin Kolb, Tobias Weich, and Lasse L. Wolf}
\email{kolb@math.uni-paderborn.de, weich@math.uni-paderborn.de, llwolf@math.uni-paderborn.de}
\begin{document}
\begin{abstract}
 The kinetic Brownian motion on the sphere bundle of a Riemannian manifold~$\M$ is a stochastic process that models a random perturbation  of the geodesic flow. 
 If $\M$ is a orientable compact constantly curved surface, we show that in the limit of infinitely large perturbation the $L^2$-spectrum of the infinitesimal generator of a time rescaled version of the process converges to the Laplace spectrum of the base manifold.
\end{abstract}

\maketitle

\section{Introduction}
Kinetic Brownian motion is a stochastic process that describes a stochastic perturbation of the geodesic flow and has the property that the perturbation affects only the direction of the velocity but preserves its absolute value. It has been studied in the past years by several authors in pure mathematics \cite{FLJ07, angst, Li16,alexis, BT18} but versions of this diffusion process have been developed independently as surrogate models for certain textile production processes  (see e.g. \cite{GKMW07,GS13, KSW13}). 

Kinetic Brownian motion $(Y_t^\g)_{t\geq 0}$ in the setting of a compact Riemannian manifold $(\M, g)$ can be informally described in the following way: $(Y_t^\g)_{t\geq 0}$ is a stochastic process with continuous paths described by a stochastic perturbation of the geodesic flow on the sphere bundle $S\M =\{\xi\in T\M, \|\xi\|_g=1\}.$
More precisely, if we denote the geodesic flow vector field by $X$ and the (non-negative) Laplace operator on the fibers of $S\M$ by $\Delta_\bbS$, then the kinetic Brownian motion is generated by the differential operator 
\[
 -X +\frac 12 \g \Delta_\bbS\colon L^2(S\M)\to L^2(S\M).
\]
The connection to the stochastic process $(Y_t^\g)_{t\geq 0}$ is given via
\[
e^{-t( -X +\frac 12 \g \Delta_\bbS)}f(x) = \mathbb E_x[f(Y_t^\g)] \quad\text{with}\quad f\in L^2(S\M), x\in S\M.
\]
Observe that the parameter $\g>0$ controls the strength of the stochastic perturbation and it is a natural question to study the behavior of $ -X +\frac 12 \g \Delta_\bbS$ and $Y_t^\g$ in the regimes $\g\to 0$ as well as $\g\to\infty$.  
 Drouot \cite{alexis} has studied the convergence of the discrete spectrum  of $ -X +\frac 12 \g \Delta_\bbS$ in the limit $\g\to 0$ for negatively curved manifolds and has shown that it converges to the Pollicott-Ruelle resonances of the geodesic flow.  These resonances are a replacement of the spectrum of $X$ since its $L^2$-spectrum is equal to $i\R$ and they can be defined in various generalities of hyperbolic flows as pole of the meromorphically continued resolvent \cite{Liv04, FS11, DZ16a, DG16, DR16,BW17}.  
 A more general framework of semiclassical subelliptic operators that includes the kinetic Brownian motion for $\g\to 0$ has been established by Smith \cite{Smith}. 
In the limit of large random noise Li \cite{Li16} and Angst-Bailleul-Tardif \cite{angst} proved  that $\pi(Y_{\g t}^\g)$ converges weakly to the Brownian motion on $ \M$ with speed 2 as $\g\to \infty$ where $\pi\colon S\M\to \M$ is the projection. 
This rescaled kinetic Brownian motion is generated by $P_\g =-\g X +\frac 12 \g^2 \Delta_\bbS$ whereas the Brownian motion on the base manifold is generated by the Laplace operator $\frac 12 \Delta_\M$. 
Therefore, one may conjecture that the discrete spectrum of $P_\g$ converges to the Laplace spectrum. 
We will give a proof of this fact in the case of constant  curvature surfaces:
 \begin{theorem}\label{thm:evofPg}
Let $(\M, g)$ be an orientable  compact surface of constant curvature. For every $\eta\in\sigma(\Delta_\M)$ with multiplicity $n$ there is   an analytic function $\lambda_ \eta \colon ]r_\eta,\infty[ \to \C$ such that $\lambda_\eta(\gamma)$ is an eigenvalue of $P_\gamma$ with multiplicity at least $n$ and  $\lambda_\eta (\g) \to \eta$ as $\g\to\infty$.
\end{theorem}

Note that this theorem does not imply that in a compact set all eigenvalues of $P_\g$ are close to eigenvalues of the Laplacian (see Remark~\ref{rmk:problemsforstronger} for a discussion of the problems that prevent us from proving this stronger statement).

Another question to ask is whether the kinetic Brownian motion converges to equilibrium, i.e. 
\[\mathbb E_x[f(Y_{\g t}^\g)] \stackrel{t\to\infty}\longrightarrow \int_ {S\M}f.
\]
Baudoin-Tardif \cite{BT18} showed exponential convergence, i.e.
\[\left \|e^{-t  P_\g} f- \int_{S\M} f \right \|\leq C e^{-C_\g t} \left \|f-\int_{S\M}f\right \|, \quad f\in L^2(S\M).\]
We should point out that the given rate $C_\g$ converges to 0 as $\g\to\infty$ but they conjecture that the optimal rate converges to the spectral gap of $\Delta_\M$ which is the smallest non-zero Laplace eigenvalue $\eta_1$ (see \cite[Section 3.1]{BT18}). 
A direct consequence of Theorem \ref{thm:evofPg} shows that the optimal rate $C_\g$ is less than $\Re \lambda_{\eta_1}(\g)$ for  surfaces of constant  curvature. Hence $\limsup_{\g\to \infty} C_\g\leq \eta_1$.
For a more explicit study of the convergence towards equilibrium we prove a spectral expansion and explicit error estimates in the case of constant negative curvature in \cite{kww19}.

Note that a problem related to the kinetic Brownian motion in $S\M$ is the study of the hypoelliptic Laplacian on $T\M$ introduced by Bismut \cite{Bis05}. Like the kinetic Brownian motion the hypoelliptic Laplacian interpolates between the geodesic flow and the Brownian motion. In \cite[Chapter 17]{BL08} Bismut and Lebeau prove the convergence of the spectrum of the hypoelliptic Laplacian to the spectrum of the Laplacian on $\M$ using semiclassical analysis. It seems plausible that their techniques can also be transferred to the setting of kinetic Brownian motion and might give the spectral convergence without any curvature restriction. The purpose of this article is however not to attack this general setting but show that under the assumption of constant  curvature  allows to drastically reduce the analytical difficulties. In fact we are able to reduce the problem to standard perturbation theory. This is also the reason why we are able to obtain the explicit error estimates in \cite{kww19}. 

Let us give a short outline of the proof of Theorem~\ref{thm:evofPg}: By the assumption of constant  curvature we have a three-dimensional Lie algebra $\mf g = \langle X,X_\perp,V\rangle_\C$ of vector field on $S\M$. Denoting the Gaussian curvature by $K$, the operator $\Omega = -X^2-X_\perp^2-KV^2$ commutes with $\mf g$ and $P_\g$ and has discrete spectrum.
Hence, we can decompose the corresponding $L^2(S\M)$ into eigenspaces of $\Omega$. The generator $P_\g$ preserves this decomposition  of $L^2(S\M)$ and we can study the restriction of $P_\g$ on each occurring eigenspace separately. In each of these eigenspaces the spectral asymptotics of $P_\gamma$ can then be handled by standard perturbation theory of an operator family of type (A)  in the sense of Kato. For the calculations it will be important that each eigenspace of $\Omega$ can be further split into the eigenspaces of the vector field $V$ which correspond to the Fourier modes in the fibers of $S\M\to \M$.

The article is organized as follows:
We will give a short overview over the kinetic Brownian motion and the connection between constant curvature surfaces and the global analysis of sphere bundles of constant curvature surfaces  in Sections \ref{sec:kbb} and \ref{sec:surfaces}.
After that we will recall a few results of perturbation theory for unbounded linear operators (Section \ref{sec:pertth}) which are mostly taken from \cite{kato}. 
In the limit $\g\to\infty$ one would like to consider the geodesic vector field as a perturbation of the spherical Laplacian. 
The major difficulty is that $\frac{1}{\g}X$ is no small perturbation in comparison with $\Delta_\bbS$. 
After the spectral decomposition with respect to $\Omega$ there is a precise way to consider $X$ as small operator in any eigenspace of $\Omega$.
Afterwards we will give the proof of the convergence of the spectra (Section~\ref{sec:der}). 

{\bf Acknowledgements} We want to thank the anonymous referee for helpful comments that led to a much clearer form of the article. 
T. Weich acknowledges the support by the Deutsche Forschungsgemeinschaft (DFG) through the Emmy Noether group “Microlocal Methods for Hyperbolic Dynamics”(Grant No. WE 6173/1-1).

\section{Preliminaries}\label{sec:preliminaries}
\subsection{Kinetic Brownian Motion}\label{sec:kbb}
Let $\M$ be a compact Riemannian manifold of dimension $d\geq 2$ with sphere bundle $S\M=\{(x,v)\in T\M\mid \|v\| =1\}$. We introduce the spherical Laplacian $\Delta_\bbS$ as follows: for every $x \in \M$ the tangent space $T_x\M$ is a Euclidean vector space via the Riemannian metric and $S_x\M=\{v\in T_x\M\mid \|v\|=1\}$ is a submanifold of $T_x\M$. The inner product on $T_x\M$ induces a Riemannian structure on $S_x\M$. Hence, the (positive) Laplace-Beltrami operator $\Delta_\bbS (x)\coloneqq \Delta_{S_x\M}$  of $S_x\M$ defines an operator $C^\infty(S_x\M)\to C^\infty (S_x\M)$. We now obtain the spherical Laplace operator $\Delta_\bbS$ by  
\[\Delta_\bbS: C^\infty(S\M)\to C^\infty(S\M),\quad \Delta_\bbS f (x,v):= (\Delta_\bbS(x)f(x,\cdot))(v).\]

For $(x,v)\in S\M$ and $w \in T _{(x,v)} S\M$ we define $\theta_{(x,v)} (w) =g_x (v, T_{(x,v)} \pi \,w)$ where $\pi \colon S\M\to \M$ is the projection and $g$ is the Riemannian metric on $\M$. Then $\theta$ is a 1-form on $S\M$ and $\nu =\theta \wedge (d\theta)^{d-1}$ defines the Liouville measure on $S\M$ which is invariant under the geodesic flow $\phi_t$. The vector field $X = \d \phi_t^\ast$ is called the geodesic vector field.

Let us consider the operator $P_\g = -\g X + \frac 12 \g^2 \Delta_\bbS$ with domain $\dom(P_\g)=\{u\in L^2(S\M) \mid P_\g u\in L^2(S\M)\}$ for $\g > 0$. Note that the action of $P_\g $ has to be interpreted in the sense of distributions. We first want to collect some  properties of $P_\g$.

\begin{proposition}
\label{prop:kbb}
 $P_\g$ is a hypoelliptic operator with 
 $$\norm{f}_{H^{2/3}}\leq C(\norm f_{L^2}+\norm {P_\gamma f}_{L^2})\quad \text{for}\quad f\in \dom (P_\gamma).$$
  $P_\gamma$ is accretive (i.e. $\Re\langle P_\g f ,f\rangle \geq 0$) and coincides with the closure of $P_\g|_{C^\infty}$.
 Therefore, $P_\g$ has compact resolvent on $L^2(S\M)$, discrete spectrum with eigenspaces of finite dimension, and the spectrum is contained in the right half plane. $P_\g$ generates a positive strongly continuous contraction semigroup $e^{-tP_\g}$.
\end{proposition}
 \begin{proof}
  See Appendix.
 \end{proof}
\subsection{Surfaces of Constant Curvature}\label{sec:surfaces}
Let $\M$ be a orientable compact  Riemannian manifold of dimension 2 and constant curvature and let $K$ be the Gaussian curvature. 
 Since $\M$ has finitely many connected components, let us  assume without loss of generality that $\M$ is connected. 
 We follow the notation of \cite{paternain2014}. Let $X$ be the geodesic vector field on $S\M$ , and let $V$ be the vertical
vector field so that $\Delta_\bbS=-V^2$. We define $X_\perp = [X, V ]$. We then have the commutator relations  $X=[V,X_\perp]$ and $[X,X_\perp]=-KV$. In particular, $\mf g\coloneqq \C X\oplus \C X_\perp\oplus \C V$ is a Lie algebra.
The Casimir operator $\Omega$ is defined as $\Omega = -X^2-X_\perp^2-KV^2$ and it is routine to check that $$[\Omega, X] = [\Omega , X_\perp] = [\Omega, V] = 0$$ using the above commutator relations. The Laplace operator $\Delta_{S\M}$ of $S\M$ for the metric which is declared by the requirement that the frame $\{X,X_\perp,V\}$ is an orthonormal basis (i.e.  $\Delta_{S\M}= -X^2-X_\perp ^2 -V^2$) is an elliptic operator on the compact manifold $S\M$ and hence it as discrete spectrum with eigenvalues of finite multiplicity. Since $\Delta_{S\M}$ is non-negative, each eigenvalue is non-negative and the eigenspaces are orthogonal.  Both operators $\Omega$ and $V$ leave these eigenspaces invariant. Thus we have the following decomposition:
$$L^2(S\M)= \bigoplus_{k\in \Z,\eta\in \sigma(\Omega)} V_{\eta,k}$$ where $V_{\eta,k} = \{u\in C^\infty(S\M)\mid Vu=ik u, \quad \Omega u =\eta u\}$ is finite-dimensional and the sum is countable. Since $V$ is skew-symmetric, the decomposition is orthogonal. The subspace $V_\eta = \bigoplus_{k\in \Z} V_{\eta,k}$ is the eigenspace of $\Omega$ which is invariant under all three vector fields $X,X_\perp,V$ and in particular invariant under $\Delta_\bbS$.

\begin{remark}
 Note that $\mf g$ is isomorphic to the complexification of $\mf {so}(3)$, $\R^2\rtimes \mf{so}(2)$, $\mf {sl}_2(\R)$ if $K>1$, $K=0$, $K<0$ respectively. We are essentially decomposing the representation of $\mf g$ on $L^2(S\M)$ into irreducible ones. 
 In fact in all three cases $S\M$ can be written as $\G\backslash G$ for some torsion free, discrete, cocompact subgroup $\G\subseteq G$ with $G\in \{PSL_2(\R), \R^2\rtimes SO(2), SO(3)\}$. The decomposition $L^2(S\M)=\bigoplus V_\eta$ can be seen  as a Plancherel decomposition of this space and $V_\eta =\bigoplus V_{\eta,k}$ as the decomposition into $K$-types or weights respectively.
 In all three cases the irreducible representations have explicit realizations on certain $L^2$-spaces (see e.g. \cite[Ch.~8]{taylor} for $\mf {sl}_2$) and one could go on by analyzing those but they do not contain more information than the abstract decomposition we provided here for all three cases at once.  
 We would like to note that this harmonic analysis point of view was our original approach motivated by previous works that used similar techniques for geodesic flows \cite{FF03, DFG15, GHW18, GHW18a, KW17}.
\end{remark}

Let us furthermore define $X_\pm \coloneqq \frac 12 (X\pm iX_\perp)$. We then have the commutator relations $$[V,X_\pm]=\pm i X_\pm \quad \text{and}\quad [X_+,X_-] = \frac 12 iKV.$$ Hence, $X_\pm \colon V_{\eta,k}\to V_{\eta,k\pm 1}$. Moreover, $X_\pm^\ast = - X_\mp$ and $\Omega = -2X_+X_--2X_-X_+-KV^2=-4X_+X_-+iKV-KV^2$.

The next lemma is crucial for our main result as  it connects the spectral values $\eta$ of $\Omega$ to the spectrum of the Laplace operator of the base manifold $\M$.
\begin{lemma}\label{la:furtherdecomp}
 Let $\Delta_\M$ be the Laplace-Beltrami operator of $\M$. Then $$\sigma(\Delta_\M)=\{\eta\in\sigma(\Omega)\mid V_{\eta,0}\neq 0\}=\sigma(\Omega)\cap \R_{\geq 0}.$$
 Moreover, $\mf g$ acts trivially on $V_{0,0}$ and  for $\eta\in \sigma(\Delta_\M)$, $\eta>0$, there is a closed $\mf g$-invariant subspace $V_\eta '$ of $V_\eta$ such that $V_{\eta,k}'\coloneqq V_{\eta,k} \cap V_\eta ' $ satisfies $\dim V_{\eta,k} ' \leq 1$ and $\dim V_{\eta,0} ' = 1$ and $V_\eta$ is isomorphic to $m_\eta$ orthogonal copies of $V_\eta'$ where $m_\eta$ is the multiplicity of the eigenvalue $\eta$ of $\Delta_\M$.
\end{lemma}
\begin{proof} Let $\eta\in \sigma(\Omega)$ with $V_{\eta,0}\neq 0$. On $V_{\eta,0}\neq 0$ the operator $\Omega$ equals the non-negative operator $\Delta_{S\M}$ and acts by $\eta$. Therefore,  $\eta\geq 0$. Conversely, pick $\eta\geq 0 $ in $\sigma(\Omega)$, i.e. $V_{\eta,k_0}\neq 0$ for some $k_0 \in \Z$. If $k_0=0$ we are done. For $\eta=0$ the space $V_{0,0}$ consists of constant functions and therefore $V_{0,0}$ is non-zero. Equation~\eqref{eq:normXpm} below shows that $X_\pm$ vanish on $V_{0,0}$ and therefore $V_{0,0}$ is a trivial $\mf g$-space. For $\eta>0$ consider the operators $X_\pm X_\mp$: We have 
 $$X_+X_- = -\frac 14 (\eta +Kk-Kk^2)\quad \text{and}\quad X_-X_+ = -\frac 14 (\eta-Kk-Kk^2)$$ are scalar on $V_{\eta,k}$. Since $X_\pm^\ast = - X_\mp$ we have \begin{align}
 \|X_+ u \|^2 = \frac 14 (\eta-Kk-Kk^2)\|u\|^2 \quad \text{and}\quad \|X_-u\|^2=\frac 14 (\eta +Kk-Kk^2)\|u\|^2 \label{eq:normXpm}
 \end{align}
for $u\in V_{\eta,k}$. In particular,  if $V_{\eta,k}\neq 0$ then $\eta\mp Kk-Kk^2 \geq 0$ and if $\eta\mp Kk-Kk^2 >0$ then $X_\pm\colon V_{\eta,k}\to V_{\eta,k\pm1}$ is injective.
More specifically, $X_\pm$ are both injective for all $k\in \Z$ in the case of $K\leq 0$. This implies that $V_{\eta,0}\neq 0$ in this case. 

If $K$ is positive, we may assume $k_0>0$. The case $k_0<0$ is handled similarly. As mentioned above it holds that $\eta-  Kk_0-Kk_0^2 \geq 0$. Analyzing the quadratic equations we observe $ \eta+  Kk-Kk^2 >0$ for $k=1,\ldots,k_0$. Hence 
$$\xymatrix{V_{\eta,k_0}\ar[r]^{X_-}&V_{\eta,k_0-1}\ar[r]^{X_-}&\cdots \ar[r]&V_{\eta,1}\ar[r]^{X_-}&V_{\eta,0}}$$ are all injective. We infer $V_{\eta,0}\neq 0$.

 Pick $f\in V_{\eta,0}$ and define $\mc H_f$ as closure of the $\mf g$-invariant subspace generated by $f$. Since $X_\pm X_\mp$ are scalar on $V_{\eta,k}$ we see that $\mc H_f =\ov{\operatorname{span}}  \{X_+^kf,X_-^k f\mid k\in\N_0\}$ and it follows that $\mc H_f\cap V_{\eta,k}$ is at most one-dimensional. Moreover, $\mc H_f$ for different $f\in V_{\eta,0}$ are isomorphic as $\mf g$-spaces  and we claim  $\mc H_g\perp \mc H_f$ for $g\perp f$.
 Since $V_{\eta,k}$ are orthogonal for different $k$ we only have to verify $(\mc H_f\cap V_{\eta,k})\perp (\mc H_g\cap V_{\eta,k})$. But as this subspace is given by $\C X_{\sign k} ^{|k|} f$ and  $\C X_{\sign k} ^{|k|} g$ respectively we need to show $\langle X_\pm^l f, X_\pm l g>=0$ for $l\geq 0$. As $\langle X_\pm^l f, X_\pm ^l g> = (-1)^l \langle X_\mp^l X_\pm^l f,g\rangle$ and $X_\pm X_\mp$ is scalar on $V_{\eta,k}$ the claim follows.
 
 If we  pick   an orthonormal basis $f_1,\ldots,f_n$ of $V_{\eta,0}$, then $\bigoplus \mc H_{f_i} \subseteq V_\eta$ and the above argument shows that equality holds. Hence, we can choose $V_{\eta}' = \mc H_{f_1}$.
 
 It remains to verify that $\Omega =\Delta_\M$ on $V_{\eta,0}\subseteq L^2(\M)$. Since $V=0$ on this space  we need to calculate $-4X_+X_-$. We use isothermal coordinates, i.e. coordinates $(x,y)$ such that the metric on $\M$ is given by $ds^2 = e^{2\lambda}(dx^2+dy^2)$. Then $\Delta_\M = -e^{-2\lambda }\left(\frac{\partial^2}{\partial x^2}+\frac{\partial^2}{\partial y^2}\right)$. Furthermore, we have (see \cite{paternain2014}):
 \begin{align*}
X_+(u) &= e^{(k-1)\lambda}\partial(h e^{-k\lambda})e^{i(k+1)\theta} \quad\mbox{ and }\\
X_-(u) &= e^{(-k-1)\lambda}\overline\partial(h e^{k\lambda})e^{i(k-1)\theta}
\end{align*}
where $u(x,y,\theta)=h(x,y)e^{ik\theta}\in V_{\eta,k}$ and $\theta$ is the angle between a unit vector and $\frac{\partial}{\partial x}$.
With this notation we have for $h\in V_{\eta,0}$
\[
X_+X_-h=X_+ (e^{-\lambda}\overline{\partial}(h)e^{-i\theta})=e^{-2\lambda}\partial\overline{\partial}(e^{-\lambda}he^\lambda)= e^{-2\lambda}\partial\overline{\partial}h =  -\frac 14 \Delta_\M h.
\]
This completes the proof.
\end{proof}

\subsection{Perturbation Theory}\label{sec:pertth}

We want to collect some basic results from perturbation theory for linear operators that can be found in \cite{kato}.
First, we introduce families of operators we want to deal with.

\begin{definition}[{see \cite[Ch. VII \S2.1]{kato}}]
 A family $T(x)$ of closed operators on a Banach space $X$ where $x$ is an element in a domain $D\subseteq \C$ is called \emph{holomorphic of type (A)} if the domain of $T(x)$ is independent of $x$ and $T(x)u$ is holomorphic for every $u\in\operatorname{dom}(T(x))$.  
\end{definition}

Without loss of generality let us assume that 0 is contained in the domain $D$. We call $T=T(0)$ the unperturbed operator and $A(x)=T(x)-T$ the perturbation. Furthermore, let $R(\zeta,x)= (T(x)-\zeta)\inv$ be the resolvent of $T(x)$ and $R(\zeta)=R(\zeta,0)$. If $\zeta\notin \sigma(T)$ and $1+A(x)R(\zeta)$ is invertible then $\zeta \notin \sigma(T(x))$ and the following identity holds:
\begin{align}\label{eq:resolventformula2}R(\zeta,x)=R(\zeta)(1+A(x)R(\zeta))\inv.\end{align}

Let us assume that $\sigma(T)$ splits into two parts by a closed simple $C^1$-curve $\G$. Then there is $r>0$ such that $R(\zeta,x)$ exists for $\zeta\in \G$ and $|x|<r$ (see \cite[Ch. VII Thm. 1.7]{kato}).
If the perturbation is linear (i.e. $T(x)=T+xA$) then a possible choice for $r$ is given by $\min_{\zeta\in\Gamma}\|AR(\zeta)\|^{-1}$. Note that $AR(\zeta)$ is automatically bounded by the closed graph theorem.
In particular, we obtain that $\Gamma\subseteq \C\setminus\sigma(T(x))$ for $|x|<r$, i.e. the spectrum of $T(x)$  still splits into two parts by $\Gamma$. Let us define $\sigma_{\interior}(x)$ as the part of $\sigma(T(x))$ lying inside $\Gamma$ and $\sigma_{\ext}(x)=\sigma(T(x))\setminus\sigma_{\interior}(x)$. 
The decomposition of the spectrum gives a $T(x)$-invariant decomposition of the space $X= M_{\interior}(x)\oplus M_{\ext}(x)$ where $M_{\interior}(x)=P(x)X$ and $M_{\ext}(x)=\ker P(x)$ with the bounded-holomorphic projection \[P(x)=-\frac{1}{2\pi i}\int_\Gamma R(\zeta,x)d\zeta.\] 
Furthermore, $\sigma(T(x)|_{M_{\interior}(x)})=\sigma_{\interior}(x)$ and $\sigma(T(x)|_{M_{\ext}(x)})=\sigma_{\ext}(x)$. 
To get rid of the dependence of $x$ in the space $M_{\interior}(x)$ we will use the following proposition.

\begin{proposition}[see {\cite[Ch. II \S4.2]{kato}}]
Let $P(x)$ be a bounded-holomorphic family of projections on a Banach space $X$ defined in a neighbourhood of 0.
 Then there is a bounded-holomorphic family of operators $U(x)\colon X\to X$ such that $U(x)$ is an isomorphism for every $x$ and $U(x)P(0)=P(x)U(x)$. In particular, $U(x) P(0) X = P(x)X$ and $U(x)\ker P(0) = \ker P(x)$.
\end{proposition}

Denoting $U(x)^{-1} T(x) U(x)$ as $\widetilde{T}(x)$ we observe \[\sigma(\widetilde T(x)|_{M_{\interior}(0)})= \sigma(\widetilde T(x)) \cap \operatorname{int}(\G) = \sigma(T(x)) \cap \operatorname{int}(\G)\] since $U(x)$ is an isomorphism. Here we denote the interior of $\G$ by $\operatorname{int}( \G)$.

Let us from now on suppose that $\Gamma$   encloses an eigenvalue $\mu$ of $T$ with finite multiplicity  and  no other eigenvalues of  $T$. Then $\sigma_{\interior}(0)=\{\mu\}$  and $M_{\interior}(0)$ is finite dimensional.
Hence, $\widetilde T(x)|_{M_{\interior}(0)}$ is a holomorphic family of  operators on a finite dimensional vector space. It follows that the eigenvalues of $T(x)$ are continuous as a function in $x$.
In addition to the previous assumptions, let us suppose that the eigenvalue $\mu$ is simple. 
Then $M_{\interior}(0)$ is one-dimensional and $\widetilde T(x)|_{M_{\interior}(0)}$ is a scalar operator.
We obtain that there is a holomorphic function $\mu\colon B_r\to\C$ (with $r=\min_{\zeta\in\Gamma}\|AR(\zeta)\|^{-1}$ as above) such that $\mu(x)$ is an eigenvalue of $T(x)$, $\mu(x)$ is inside $\Gamma$ and $\mu(x)$ is the only part of $\sigma(T(x))$ inside $\Gamma$ since $\sigma_{\interior}(x)=\sigma(\widetilde T(x)|_{M_{\interior}(0)})$.

We now want to calculate the Taylor coefficients of $\mu(x)$ in order to get an approximation of $\mu(x)$  in the case where $X=\mc H$ is a Hilbert space and $T(x)$ is a holomorphic family of type (A) with symmetric $T$ but not necessarily symmetric $T(x)$ for $x\neq 0$.
To this end let $\varphi(x)$ be a normalized holomorphic family of  eigenvectors (obtained from $P(x)$).
Consider the Taylor series $\mu(x)=\sum  x^n \mu^{(n)}$,  $\varphi(x)=\sum  x^n \varphi^{(n)}$ and  $T(x)u=\sum x^n T^{(n)} u $ for every $u \in \dom(T)$ which converges on a disc of positive radius independent of $u$. This is due to the fact that Taylor series of holomorphic functions converge on every disc that is contained in the domain.

We compare the Taylor coefficients in \begin{align*}
(T(x)-\mu(x))\varphi(x)=0\qquad \text{and} \qquad\langle(T(x)-\mu(x))\varphi(x),\varphi(x)\rangle=0
\end{align*}
and obtain \begin{equation*}
(T-\mu^{(0)})\varphi^{(l)}=-\sum_{n=1}^l(T^{(n)}-\mu^{(n)})\varphi^{(l-n)}\end{equation*} and
\begin{align*}
\mu^{(k)}= &\langle T^{(k)}\varphi^{(0)},\varphi^{(0)}\rangle+\sum_{n=1}^{k-1}\langle(T^{(n)}-\mu^{(n)})\varphi^{(k-n)}, \varphi^{(0)}\rangle.
\end{align*}
A fortiori,
\begin{align}
\mu^{(1)}&= \langle T^{(1)}\varphi^{(0)},\varphi^{(0)}\rangle \label{eq:firstderivative} \\ 
\mu^{(2)}&= \langle T^{(2)}\varphi^{(0)},\varphi^{(0)}\rangle+\langle(T^{(1)}-\mu^{(1)})\varphi^{(1)}, \varphi^{(0)}\rangle, \label{eq:secondderivative} \end{align}
where $\varphi^{(1)}$ fulfils \begin{equation}(T-\mu^{(0)})\varphi^{(1)}=-(T^{(1)}-\mu^{(1)})\varphi^{(0)}.\label{eq:derivativevector}
 \end{equation}
  For $\varphi^{(1)}$ being  uniquely determined we can use the additional assumption that $\varphi(x)$ is normalized. However  $\mu^{(2)}$ can be calculated without this consideration in our setting.
Here $\varphi^{(1)}=v+c\varphi^{(0)}$ with unique $v\in \ker(T^{(0)}-\mu^{(0)})^\perp$ as $T^{(0)}$ is symmetric. 
We infer that 
\begin{align*}
 \langle(T^{(1)}-\mu^{(1)})\varphi^{(1)}, \varphi^{(0)}\rangle&=\langle(T^{(1)}-\mu^{(1)})v, \varphi^{(0)}\rangle-c\mu^{(1)}+c\langle T^{(1)}\varphi^{(0)}, \varphi^{(0)}\rangle \\&= \langle(T^{(1)}-\mu^{(1)})v, \varphi^{(0)}\rangle.
\end{align*}
 Therefore, $\mu^{(2)}$ depends only on $v$ and not on $c$.

\section{Perturbation Theory of the Kinetic Brownian Motion}
\label{sec:der}

 We want to establish the limit $\gamma\to \infty$ of the spectrum of $P_\g$. 
 To do so we write $P_\g= \frac{\g^2}2 (\Delta_\bbS-2\g^{-1} X )=\frac{\g^2}2 T(-2\g^{-1})$ where  $T(x)=\Delta_\bbS+xX$ and we want to use the methods established in Chapter \ref{sec:pertth}. 
 
In order to have finite dimensional eigenspaces and holomorphic families of type (A) we will use the orthogonal eigenspace decomposition of $L^2(S\M)$ derived in Section \ref{sec:surfaces}:
\[L^2(S\M) = \bigoplus_{\eta,k}V_{\eta,k}\]
where $V_{\eta,k}=\{u\in C^\infty(S\M)\mid \Omega u= \eta u, Vu=iku\}$.

\begin{proposition}
 The family of operators $T(x), \, x\in \C,$ restricted to $V_\eta$ defines a holomorphic family of type (A) with domain $H^2(S\M)\cap V_\eta$. The same is true for $V_{\eta}'$.
\end{proposition}
\begin{proof}
 Since $\Delta_{S\M }$ is a second-order elliptic differential operator, we have $H^2(S\M)=\{u\in L^2(S\M)\mid (\Delta_{S\M} +xX) u \in L^2(S\M)\}$ for each $x\in \C$.
 The space $V_\eta$ is invariant under $\Delta_{S\M}$ so that $H^2(S\M)\cap V_\eta = \{u\in V_\eta \mid (\Delta_{S\M} +xX) u \in V_\eta \}$. We now use that $\Delta_{S\M} = \Omega +(1-K)\Delta_\bbS$ and $\Omega=\eta $ on $V_\eta$. Therefore, $H^2(S\M)\cap V_\eta = \{u\in V_\eta \mid (1-K)\Delta_\bbS +xX) u \in V_\eta \} = \dom(T(x/(1-K))|_{V_\eta})$ for $K\neq 1$. For $K=1$ the same argument works if we replace $\Delta_{S\M}$ by $\Delta_{S\M}+\Delta_\bbS$.  Since $T(x)|_{V_\eta}$ is closed as a restriction of a closed operator the proposition is proven. The proof for $V_{\eta}'$ is identically.
\end{proof}

We denote the restriction of $T(x)$ to $V_\eta$ by $T_\eta(x)$.

The eigenspaces of the unperturbed operator $\Delta_\bbS|_{V_\eta}$ are $V_{\eta,0}$ and $V_{\eta,k}\oplus V_{\eta,-k}$ which are finite dimensional. 
As we have seen in Section~\ref{sec:pertth} the eigenvalues of a holomorphic family of type (A) are continuous as a function of $x$ in this case. We deduce that for the eigenvalues $\mu(x)$ of $T_\eta(x)$ that arise from non-zero eigenvalues $\mu = \mu(0)$ of $\Delta_\bbS|_{V_\eta}$ the limit $\gamma\to \infty$ of $\frac{\gamma^2}2\mu(2 \gamma\inv)$, which is an eigenvalue of $P_\gamma$, is $\infty$.
Therefore, we are only interested in eigenvalues of $T_\eta(x)$ which arise from the unperturbed eigenvalue 0. 
In order to have that $0$ is an eigenvalue of $T_\eta(0)$ we must have $V_{\eta,0}\neq 0$, i.e. $\eta \in\sigma(\Delta_\M)$ by Lemma~\ref{la:furtherdecomp}.
Let us first deal with $\eta=0$. Here $\mf g$ acts trivially on $V_{0,0}$ by  Lemma~\ref{la:furtherdecomp} and therefore the eigenvalue of  $T_0(x)$ that arises from the eigenvalue 0 is 0. If $\eta>0$ we restrict $T_\eta(x)$ to $V_\eta'$ which defines a holomorphic family of type (A) as well.
In this case the eigenspace of the unperturbed eigenvalue is $V_{\eta,0}'$ which is one-dimensional. Hence, we are in the precise setting of Section~\ref{sec:pertth}. We obtain that there is a holomorphic function $\mu$ defined on a neighbourhood of 0 (depending on $\eta$) such that $\mu(x)$ is an eigenvalue of $T(x)|_{V_\eta'}$ with $\mu(0)=0$.

Let $\varphi(x)$ be a corresponding holomorphic normalized eigenvector, in particular $\varphi(0)\in V_{\eta,0}'$.
 We can use Equation \eqref{eq:firstderivative} from Section \ref{sec:pertth}:  
\begin{align*}
\mu'(0)=\langle X \varphi(0), \varphi(0)\rangle= \frac 12\langle (X_++X_-) \varphi(0),\varphi(0)\rangle.
\end{align*}
 Due to the fact that $X_\pm$ are raising respectively lowering operators, i.e. $X_\pm V_{\eta,k}\subseteq V_{\eta,k\pm1}$, we conclude that $\mu'(0)=0$. 

We now want to find the second derivative $\mu''(0)$ of $\mu$. 
According to Section \ref{sec:pertth} we first have to compute $\varphi'(0)$ via $\Delta_\bbS\varphi'(0)=-X\varphi(0)$ (see Equation~ \eqref{eq:derivativevector}). 
Notice that $X\varphi(0)\in V_{\eta,-1}'\oplus V_{\eta,1}' = \{u\mid \Delta_\bbS u =u\}$. Furthermore $\ker(\Delta_\bbS|_{V_\eta'})=V_{\eta,0} '$, and consequently  $\varphi'(0)=-X\varphi(0)+\phi_0$ for some $\phi_0 \in V_{\eta,0}'$. Let us recall that $\mu''(0)$ is independent of $\phi_0$. 
Consequently by Equation \eqref{eq:secondderivative}, \begin{align*}
\mu''(0)  =& 2\langle X(-X\varphi(0)), \varphi(0)\rangle=-\frac{1}{2}\langle (X_++X_-)^2\varphi(0),\varphi(0)\rangle\\ =& -\frac12 \langle ( X_+^2+ X_+X_-+ X_-X_++ X_-^2)\varphi(0),\varphi(0)\rangle.
\intertext{Again, $X_\pm$ are raising/lowering operators. Therefore, }
\mu''(0)&=-\frac12\langle (X_+X_- +X_-X_+)\varphi(0),\varphi(0)\rangle \\
&=\langle \Omega \varphi(0),\varphi(0)\rangle = \eta
\end{align*}
as the Casimir operator $\Omega$ equals $-2X_+X_--2X_-X_+-KV^2$ and $V \varphi(0)=0$.

If we now substitute $x=-2\g\inv$ and   observe that $\lambda_\eta(\g)=\frac {\g^2} 2 \mu(-2\g\inv)$ is an eigenvalue of $P_\g$ we obtain Theorem~\ref{thm:evofPg}.

\begin{remark}\label{rmk:problemsforstronger}
 In order to obtain  uniform convergence  of the eigenvalues in compact sets we would like to deal with all $\eta\in \sigma(\Delta_\M)$ simultaneously in a uniform way. More precisely, we want to separate converging eigenvalues (which arise from $0$) from non-converging eigenvalues. For this to happen we must have that $1+xX(\Delta_\bbS-\zeta)\inv$ on $V_\eta'$ is invertible for small $|x|$ and  for $\zeta$ in some closed curve enclosing $ 0$ but no other element of $\sigma(\Delta_\bbS)=\{k^2\mid k\in\Z\}$. In particular, $1+xX(\Delta_\bbS-\zeta)\inv$ has to be invertible for some $\zeta \in (0,1)$ but we can only ensure this for $|x|<\|X(\Delta_\bbS-\zeta)\inv|_{V_\eta'}\|\inv$. As 
 \begin{align*}\|X(\Delta_\bbS-\zeta)\inv|_{V_\eta'}\| &\geq\|X(\Delta_\bbS-\zeta)\inv|_{V_{\eta,0}'}\|= |\zeta|\inv\|X|_{V_{\eta,0}'}\|=|\zeta|\inv\sqrt{\|X_+|_{V_{\eta,0}'}\|^2+\|X_-|_{V_{\eta,0}'}\|^2}\\
 &=|\zeta|\inv\sqrt{\frac 12 \eta}\geq \sqrt{\frac 12 \eta}
 \end{align*}
this is impossible for all $\eta\in \sigma(\Delta_\M)$ at once.
\end{remark}

\appendix
\section{Proof of Proposition \ref{prop:kbb}}
The proof that $P_\g$ is hypoelliptic with the subelliptic estimate can be found in \cite[Chapter 2.2]{alexis}. There exist vector fields $X_j$ on $S\M$ such that $\Delta_\bbS = -\sum_{j=1}^d X_j^2$ and $\operatorname{div} X_j = 0$ (see \cite[\S 2.2.6]{alexis}). Hence, the $X_j$ as well as $X$ are skew-symmetric with respect to the inner product of $L^2(S\M)$.  It follows that $\Re \langle P_\g f, f\rangle = \sum \frac 12\g^2 \langle X_j f, X_j f\rangle - \g\Re \langle X f , f \rangle\geq 0$, i.e. $P_\g|_{C^\infty}$ is accretive since $\langle X f, f\rangle \in i\R$.

We show that $\operatorname{Ran}(P_\g|_{C^\infty} + I)$ is dense following the proof of \cite[Prop.~5.5]{witten}. Let $f \in \operatorname{Ran}(P_\g|_{C^\infty} + I)^\perp$. Then we have $\langle f, (P_\g + I)u\rangle = 0$ for all $u \in C^\infty$, hence $(P_{-\g} + I)f = 0$ in $\mc D'$. Since $P_{-\g}$ is hypoelliptic, it follows that $f\in C^\infty$ and $0= \sum \frac 12 \g^2\langle X_j f,X_j f\rangle + \langle f,f \rangle - \g\langle Xf ,f\rangle$. Thus $f = 0$. 

We obtain that the closure $\overline {P_\g|_{C^\infty}}$ is maximal-accretive (see e.g. \cite[Thm.~5.4]{witten}).
An operator $A$ on a Hilbert space is maximal-accretive iff it generates a contraction semigroup $e^{-tA}$ (see \cite[p. 241]{reedsimon}) Hence, $\overline {P_\g|_{C^\infty}}$ generates a contraction semigroup $e^{-tP_\g}$. The adjoint semigroup $(e^{-tP_\g})^\ast$ is generated by $(P_\g|_{C^\infty })^\ast$ that is $\frac 12 \g^2 \Delta_\bbS + \g X$ with domain $\{f \in L^2\mid (\frac 12 \g^2 \Delta_\bbS + \g X)f \in L^2\}$ (see \cite[I.5.14 and II.2.5]{engelnagel}). In particular, this operator is maximal-accretive. In analogy we infer that both $\overline{P_\g|_{C^\infty}}$ and $P_\g$ are maximal-accretive and we conclude that they coincide. Similar arguments can be found in \cite{GS14}.

For the positivity of the generated contraction semigroup we have to check if $$\langle (\operatorname{sign}f) P_\g f, u\rangle \geq \langle |f|, (P_\g)^\ast u\rangle$$ for all real $f\in C^\infty$ and a strictly positive subeigenvector $u$ of $(P_\g)^\ast$ (see \cite[C-II Cor. 3.9]{arendt}). Note that $1$ is a strictly positive eigenvector of $(P_\g)^\ast$ and $\frac 12 \Delta_\bbS (x)$ as well as $- X$ generate stochastic Feller processes on $S_x\M$ and $S\M$ respectively (namely the Brownian motion on $S_x\M$ and the  geodesic flow). Hence, $e^{-t\Delta_\bbS(x)}$ and $e^{tX}$ define positive semigroups so that $\langle (\operatorname{sign} f) \Delta_\bbS(x)f, 1\rangle_{S_x\M}\geq 0$ for $f\in C^\infty (S_x\M)$ and $\langle (\operatorname{sign} f) (-X)f, 1\rangle\geq 0$ for $f\in C^\infty(S\M)$ (see \cite[C-II Thm.2.4]{arendt}). Combining both statements completes the proof.


\newcommand{\etalchar}[1]{$^{#1}$}
\providecommand{\bysame}{\leavevmode\hbox to3em{\hrulefill}\thinspace}
\providecommand{\MR}{\relax\ifhmode\unskip\space\fi MR }
\providecommand{\MRhref}[2]{%
  \href{http://www.ams.org/mathscinet-getitem?mr=#1}{#2}
}
\providecommand{\href}[2]{#2}

\end{document}